\documentclass[reqno]{amsart}
\usepackage{amssymb}

\newcommand{\remove}[1]{ }

\DeclareMathOperator{\card}{Card}
\newtheorem{theorem}{Theorem}[section]
\newtheorem{proposition}[theorem]{Proposition}

\theoremstyle{definition}
\newtheorem{definition}[theorem]{Definition}

\theoremstyle{remark}
\newtheorem{remark}[theorem]{Remark}
\newtheorem{remarks}[theorem]{Remarks}
\newtheorem{examples}[theorem]{Examples}
\newtheorem{example}[theorem]{Example}

\numberwithin{equation}{section}

\begin{document}
\title{Stable schedule matchings by a fixed point method}
\author{Vilmos Komornik}
\address{D\'epartement de math\'ematique\\
Universit\'e de Strasbourg\\
7 rue Ren\'e Des\-car\-tes\\
67084 Strasbourg Cedex, France}
\email{vilmos.komornik@math.unistra.fr}
\author{Zsolt Komornik}
\address{UFR de math\'ematique et d'informatique\\
Universit\'e de Strasbourg\\
7 rue Ren\'e Des\-car\-tes\\
67084 Strasbourg Cedex, France}
\email{z.komornik@gmail.com}
\author{Christelle K. Viauroux}
\address{University of Maryland Baltimore County\\
Department of Economics\\
1000 Hilltop Circle\\
Baltimore, MD 21250, USA}
\email{ckviauro@umbc.edu}
\date{Version 2010-05-11-a}
\thanks{Part of this research was realized during the stay of the first
author at the Department of Mathematics of the University of Cincinnati as a
Taft research fellow in March--June, 2007. He is grateful to the Charles
Phelps Taft Research Center for their kind invitation and for the excellent
working conditions.}

\begin{abstract}
We generalize several schedule matching theorems of Baiou--Ba\-linski (Math.
Oper. Res., 27 (2002), 485) and Alkan--Gale (J. Econ. Th. 112 (2003), 289)
by applying a fixed point method of Fleiner (Math. Oper. Res., 28 (2003),
103). Thanks to a more general construction of revealing choice maps we
develop an algorithm to solve rather complex matching problems. The
flexibility and efficiency of our approach is illustrated by various
examples. We also revisit the mathematical structure of the matching theory
by comparing various definitions of stable sets and various classes of
choice maps. We demonstrate, by several examples, that the revealing
property of the choice maps is the most suitable one to ensure the existence
of stable matchings; both from the theoretical and the practical point of
view.
\end{abstract}

\maketitle

\section{Introduction}

\label{s1}

Since the pioneering paper of Gale and Shapley \cite{GalSha1962} on \emph{%
stable matchings,} many studies have been devoted to the adaptations and the
generalizations of their algorithm. Stable matching algorithms have found
use in diverse economic applications ranging from labor markets to college
admissions or even kidney exchanges.

In these two-sided matching markets, two sets of agents have preferences
over the opposite set: on one side of the market, there are individuals
(students, interns or employees) and on the other side there are
institutions (colleges, hospitals or firms). A ``stable match'' is realized
when all agents have been matched with the opposite side such that neither
could obtain a more mutually beneficial match on their own. 

The original strict preference ordering assumptions proved to be too
restrictive for many real world problems. Following an influential
contribution of Kelso and Crawford \cite{CraKel1982}, Roth \cite{Rot1984}
made a systematic study of a more flexible approach based on \emph{choice
functions}. The monograph of Roth and Sotomayor \cite{RotSot1990} provides
an overview of the state of the art up to 1990 and it still serves as an
excellent introduction to the subject. Feder \cite{Fed1992}, Subramanian 
\cite{Sub1994} and Adachi \cite{Ada2000} discovered a close relationship
between stable matchings and fixed points of set-valued maps. Then Fleiner 
\cite{Fle2003} demonstrated that many classical results may be obtained by a
straightforward application of an old theorem of Knaster \cite{Kna1928} and
Tarski \cite{Tar1928}, \cite{Tar1955}. See also Hatfield and Milgrom \cite%
{HatMil2005} for an economically motivated presentation of the fixed point
method.

More recently, Baiou and Balinski \cite{BaiBal2002} introduced the notion of 
\emph{schedule matching} which made it possible to consider, as a part of
the contract not only the hiring of a particular worker by a particular
firm, but also the number of hours of employment of the worker in the firm.
In their setting, ``stability'' means that no pair of opposite agents can
increase their hours together either due to unused capacity or by giving up
hours with less desirable partners. They assumed that all agents have strict
preference orderings. Alkan and Gale \cite{AlkGal2003} extended their model
by using incomplete revealed preference ordering via choice functions
instead.

In this paper, we generalize the notion of schedule matching of Baiou and
Balinski \cite{BaiBal2002} to allow for schedule and preference constraints
on each side of the market. We define a revealing choice map for each agent
on the acceptable opposite side agent(s), possible days and (combinations
of) restrictions or ``subsets'' placed on the opposite side agent and/or days
worked. In particular, our framework allows for possible quotas placed by
workers on firms and days worked, allowing him to work part-time for
different firms on the same day or on different days, excluding some firms
on some given days or excluding some days of work. In the same manner, it
allows firms to adjust their labor force on certain days depending on their
anticipated activity, or on the requirements associated to different
activities on different days (or the same day). We show that the allocation
of days, firms and workers is stable in the sense that given their schedule
constraints, their preference orderings and constraints, there is no better
schedule for both parties; moreover the stable allocation is shown to be
worker optimal or firm optimal. This is done by using a slightly simplified
version of Fleiner's theorem and by giving a general construction of choice
maps having the revealed preference property. We illustrate the power of our
theorems by several examples. We provide the algorithm that can be used to
obtain the optimal allocation: we will solve a deliberately complex example
to explain its technical execution. Furthermore, in order to discuss the
optimality of our results, we clarify the relationships between various
properties of choice maps and between different definitions of stable sets,
often used in the literature.

The plan of the paper is the following. In Section \ref{s2} we formulate a
model problem which will motivate our research and which may have natural
real-word applications. In Section \ref{s3} we present the mathematical
framework for our model. In Section \ref{s4} we solve the problems of
Section \ref{s2} and we also explain how our results cover some of the
theorems of Alkan and Gale \cite{AlkGal2003}. In Section \ref{s5} we
illustrate the power and flexibility of our method by solving a number of
more complex problems. Section \ref{s6} concludes. The proofs of the
theoretical results are given in Section \ref{s7}.

\section{Schedule matching problems}

\label{s2}

In order to illustrate the novelty of the present work we begin by recalling
the first example of Gale and Shapley \cite{GalSha1962}. They considered
three women: $w_1$, $w_2$, $w_3$ and three men: $f_1$, $f_2$, $f_3$ with the
following preference orders (we change the notations for consistence with
our later examples):

\begin{itemize}
\item Preference order of $w_1$: $f_1\succ f_2\succ f_3$;

\item Preference order of $w_2$: $f_2\succ f_3\succ f_1$;

\item Preference order of $w_3$: $f_3\succ f_1\succ f_2$;

\item Preference order of $f_1$: $w_2\succ w_3\succ w_1$;

\item Preference order of $f_2$: $w_3\succ w_1\succ w_2$;

\item Preference order of $f_3$: $w_1\succ w_2\succ w_3$.
\end{itemize}

They looked for the possibilities of marrying all six people in a stable
way. Instability would occur if there were a woman and a man, not married to
each other who would prefer each other to their actual mates. It turns out
that there are three solutions:

\begin{itemize}
\item each woman gets her first choice: $(w_1,f_1)$, $(w_2,f_2)$, $(w_3,f_3)$%
;

\item each man gets his first choice: $(w_1,f_3)$, $(w_2,f_1)$, $(w_3,f_2)$;

\item everyone get her or his second choice: $(w_{1},f_{2})$, $(w_{2},f_{3})$%
, $(w_{3},f_{1})$.
\end{itemize}

\bigskip 

Now let us modify the problem to a simple job market problem as follows.
Consider three workers: $w_{1}$, $w_{2}$, $w_{3}$ and three firms: $f_{1}$, $%
f_{2}$, $f_{3}$ with the same preference orders for hiring as above.
Furthermore, assume that hiring is for two different days: $d_{1}$, $d_{2}$,
with the following additional preferences and requirements:

\begin{itemize}
\item for each worker--firm pair $(w_i,f_j)$, the worker prefers $d_1$ to $%
d_2$ and the firm prefers $d_2$ to $d_1$;

\item each worker may be hired by at most one firm on each given day (maybe
different firms on different days);

\item each firm may hire at most two workers per day; if they hire two
workers for one given day, then they cannot hire anybody for the other day;

\item no firm may hire the same worker for both days.
\end{itemize}

We are looking for a stable set of contracts, i.e., for a set $S$ of
triplets $(w_i,f_j,d_k)$ having the following properties:

\begin{itemize}
\item each contract $(w_i,f_j,d_k)\in S$ is acceptable for both $w_i$ and $%
f_j$;

\item for any other contract $(w_{i},f_{j},d_{k})\notin S$, either $w_{i}$
and/or $f_{j}$ prefers her/his contracts in $S$ to this new one.
\end{itemize}

The following section presents the theory necessary to address this type of
problems. The solution to this example is given in Section \ref{s4}.

\section{Existence of stable schedule matchings}

\label{s3}

In this section we develop the theoretical framework required to solve
problems like that of the preceding section. The main results are Theorems %
\ref{t37}, \ref{t310} and \ref{t313}. Propositions \ref{p35} and \ref{p36}
contain useful complements and will be used in the proof of Theorem \ref{t37}
but they are not necessary for the understanding and the applications of our
theorems. For the reader's convenience, proofs are postponed to Section \ref%
{s6}, which contains various remarks and examples discussing the optimality
of the results formulated here.

\bigskip 

Given a set $X$, we denote by $2^{X}$ the set of all subsets of $X$. By a 
\emph{choice map} in $X$ we mean a function $C:2^{X}\rightarrow 2^{X}$
satisfying 
\begin{equation}
C(A)\subset A\quad \text{for all}\quad A\subset X.  \label{31}
\end{equation}

\bigskip 

In economic applications $X$ is the set of all possible contracts, and for a
given set $A$ of proposed contracts, $C(A)$ denotes the set of accepted
contracts by some given rules of the market.

\bigskip 

Assume that there are two competing sides, for example \emph{workers and
firms} and correspondingly two choice functions $C_{W},C_{F}:2^{X}%
\rightarrow 2^{X}$.

\begin{definition}
\label{d31} A set $S$ of contracts is said to be \emph{stable} if there
exist two sets $S_W,S_F\subset X$ satisfying the following three conditions: 
\begin{align}
&S_W\cup S_F=X;  \label{32} \\
&C_W(A)=S\quad\text{for every}\quad S\subset A\subset S_W;  \label{33} \\
&C_F(A)=S\quad\text{for every}\quad S\subset A\subset S_F.  \label{34}
\end{align}
\end{definition}

Stable contract sets represent acceptable compromises.

\begin{remark}
\label{r32} A stable set\footnote{%
A more thorough investigation of stable sets is carried out in Proposition %
\ref{p36} below.} $S$ is \emph{individually rational} if 
\begin{equation}
C_{W}(S)=S=C_{F}(S),  \label{35}
\end{equation}%
and it is \emph{not blocked by any other contract}, i.e., for each $x\in X$
we have 
\begin{equation}
\text{either}\quad C_{W}(S\cup \{x\})=S\quad \text{or}\quad C_{F}(S\cup
\{x\})=S\quad \text{(or both).}  \label{36}
\end{equation}
\end{remark}

In order to ensure the existence of stable sets of contracts we need one
additional assumption on the choice maps.

\begin{definition}
\label{d33} We say that a choice map $C:2^X\to 2^X$ is \emph{revealing} (or
satisfies the \emph{revealed preference condition}) if 
\begin{equation}  \label{37}
A,B\subset X\quad\text{and}\quad C(A)\subset B\Longrightarrow A\cap
C(B)\subset C(A).
\end{equation}
\end{definition}

This means that if a contract is rejected from some proposed set $A$ of
contracts, then it will also be rejected from every other proposed set $B$
which contains the accepted contracts.

\begin{example}
\label{e34}\mbox{}

(a) For any fixed set $Y\subset X$ the formula $C(A):=A\cap Y$ defines a
revealing choice map on $X$. This example illustrates a situation where some
contracts are unacceptable to certain agents. \medskip 

(b) More generally, given a finite subset $Y\subset X$, a nonnegative
integer $q$ (called \emph{quota}) and a strict preference ordering $%
y_{1}\succ y_{2}\succ \cdots $ on $Y$, we define a map $C(A)$ for any given $%
A\subset X$ as follows. If $\card(A\cap Y)\leq q$, then we set $C(A):=A\cap Y
$. If $\card(A\cap Y)\succ q$, then let $C(A)$ be the set of the first $q$
elements of $A\cap Y$ according to the ordering of $Y$. Then $%
C:2^{X}\rightarrow 2^{X}$ is a revealing choice map on $X$.

Choice maps of this kind are frequently used in classical matching problems
such as the marriage problem, the college admission problem and various
many-to-many matching problems; see, e.g., \cite{AlkGal2003}, \cite%
{RotSot1990} and the references of the latter.
\end{example}

Before stating our main theorem, we further clarify the relationships
between the revealed preference condition and other usual properties of
choice maps (Proposition 3.5.). We also discuss alternative equivalent
definitions of stable sets (Proposition 3.6.).

\begin{proposition}
\label{p35}\mbox{}

(a) A choice map $C:2^X\to 2^X$ is revealing if and only if it is \emph{%
consistent}: 
\begin{equation}  \label{38}
C(A)\subset B\subset A\Longrightarrow C(B)= C(A)
\end{equation}
and \emph{persistent} (or satisfies the \emph{substitute condition}): 
\begin{equation}  \label{39}
A\subset B\Longrightarrow A\cap C(B)\subset C(A).
\end{equation}

(b) A choice map is persistent if and only if the \emph{rejection map} $%
R:2^X\to 2^X$ defined by $R(A):=A\setminus C(A)$ is \emph{monotone}, i.e., 
\begin{equation}  \label{310}
A\subset B\Longrightarrow R(A)\subset R(B).
\end{equation}

(c) A choice map satisfying either \eqref{38} or \eqref{39} is \emph{%
idempotent}: 
\begin{equation}  \label{311}
C(C(A))=C(A)\quad\text{for all}\quad A\subset X.
\end{equation}
\end{proposition}

\begin{proposition}
\label{p36} We consider two choice maps $C_W, C_F:2^X\to 2^X$ and three sets 
$S, S_W, S_F\subset X$ satisfying 
\begin{equation}  \label{312}
S_W\cup S_F=X\quad\text{and}\quad C_W(S_W)=S=C_F(S_F).
\end{equation}

(a) If the choice maps $C_W,C_F:2^X\to 2^X$ are idempotent, then $S$ is
individually rational, i.e., it satisfies \eqref{35}.

(b) If at least one of the two choice maps $C_W,C_F:2^X\to 2^X$ is
consistent, then we may modify $S_W$ or $S_F$ such that that $S_W\cap S_F=S$
and \eqref{312} remains valid. \smallskip

(c) If, moreover, both choice maps are consistent, then \eqref{312} is
equivalent to the stability \eqref{32}--\eqref{34} of $S$. \smallskip

(d) If both choice maps $C_W,C_F:2^X\to 2^X$ are revealing, then a set $S$
is stable if and only if it is individually rational, and it is not blocked
by any other contract, i.e., \eqref{32}--\eqref{34} are equivalent to %
\eqref{35}--\eqref{36} (for all $x\in X$.
\end{proposition}

Observe that property \eqref{312} below follows from the definition of
stable sets. 

\bigskip 

Our main theorem below shows that the revealed preference condition ensures
the existence of stable sets of contracts:

\begin{theorem}
\label{t37} If $C_W,C_F:2^X\to 2^X$ are two revealing choice maps, then
there exists at least one stable set of contracts.
\end{theorem}

\begin{remark}
\label{r38}\mbox{}

(a) The proof of the theorem, provided in Section \ref{s7}, will show that
the stable sets form a complete lattice for a natural order relation. In
particular, there exists a worker-optimal and a firm-optimal stable set.
\medskip 

(b) In case $X$ is a finite set, the proof of the theorem provides an
efficient algorithm to find a stable set. Starting with $X_0:=X$ we compute
successively $Y_1, X_2, Y_3, X_4,\ldots $ by using the recursive formulae 
\begin{equation*}
Y_{n+1}:=(X\setminus X_n)\cup C_W(X_n) \quad\text{and}\quad
X_{n+1}:=(X\setminus Y_n)\cup C_F(Y_n).
\end{equation*}
There exists a first index $n\ge 1$ such that $X_{n-1}=X_{n+1}$, and then $%
S:=C_W(X_{n-1})$ is the worker-optimal stable set.

Similarly, starting with $Y_0:=X$ we may compute successively $X_1, Y_2,
X_3, Y_4,\ldots $ by the same recursive formulae. There exists a first index 
$n\ge 1$ such that $Y_{n-1}=Y_{n+1}$, and then $S:=C_F(Y_{n-1})$ is the
firm-optimal stable set. See Remark \ref{r62} below for the details. \medskip

(c) The definitions of revealing choice maps, stable sets, the theorem and
the preceding remarks remain valid if we replace $2^X$ by a complete
sublattice $L$ of $2^X$, i.e., a subfamily $L$ of $2^X$ such that the union
and the intersection of any system of sets $A\in L$ still belongs to $L$.
See, e.g., \cite{Fle2003} for more details on lattice properties. \medskip

(d) Part (a) of Proposition \ref{p35} shows that Theorem \ref{t37} is
mathematically equivalent to a theorem of Fleiner \cite{Fle2003}. \medskip

(e) We will show in Examples \ref{e61} (a)--(b) and \ref{e65} of Section \ref%
{s6} that the revealing condition cannot be weakened in Theorem \ref{t37}.
\end{remark}

In order to apply Theorem \ref{t37} for the solution of the problem stated
in Section \ref{s2}, we need a generalization of the construction of
revealing choice maps recalled in Example \ref{e34}. Such a construction is
provided by Theorem \ref{t310} below.

Let we are given a finite subset $Y\subset X$, a family $\{Y_n\}$ of subsets 
$Y_n\subset X$, and corresponding nonnegative integers (called \emph{quotas}%
) $q$ and $q_n$. We assume that the sets $Y_n\cap Y$ are disjoint.
Furthermore, let be given a strict preference ordering $y_1\succ y_2\succ
\cdots $ on $Y$. Given any set $A\subset X$, we define a nondecreasing
sequence $C_0(A)\subset C_1(A)\subset\cdots$ of subsets of $A\cap Y$ by
recursion as follows. First we set $C_0(A)=\varnothing$. If $C_{k-1}(A)$ has
already been defined for some $k$, then we set $C_k(A):=C_{k-1}(A)\cup\{y_k\}
$ if 
\begin{align*}
&y_k\in A, \\
&\card C_{k-1}(A) < q, \\
&\card\left( C_{k-1}(A)\cap Y_n\right) < q_n\text{ if }y_k\in Y_n;
\end{align*}
otherwise we set $C_k(A):=C_{k-1}(A)$. Finally, we define $C(A):=\cup C_k(A)$%
.

\begin{remark}
\label{r39} It follows from the construction that 
\begin{align}
&C(A)\subset A\cap Y;  \label{313} \\
&\card C(A)\le q;  \label{314} \\
&\card\left( C(A)\cap Y_n\right) \le q_n\text{ for all }n.  \label{315}
\end{align}
\end{remark}

\begin{theorem}
\label{t310} $C:2^{X}\rightarrow 2^{X}$ is a revealing choice map.
\end{theorem}

\begin{remark}
\label{r311}\mbox{}

\begin{itemize}
\item[(a)] If $q_{n}\geq q$ or $q_{n}\geq \card(Y)$ for some $n$, then we
may eliminate $Y_{n}$ and $q_{n}$ without changing the construction.

\item[(b)] If there are no sets $Y_{n}$, then our construction reduces to
Example \ref{e34} (b).

\item[(c)] If, moreover, $q\geq \card(Y)$, then our construction reduces to
Example \ref{e34} (a). (In this case the choice of the order relation is
irrelevant.)

\item[(d)] Instead of a finite subset $Y\subset X$, we can also consider
arbitrary subsets $Y\subset X$ with a well-ordered preference relation: the
construction and the proof of the proposition remain valid.
\end{itemize}
\end{remark}

\begin{example}
\label{e312} The disjointness condition is necessary. To show this, consider
the sets $X=Y=\{a,b,c\}$, $Y_{1}=\{a,b\}$, $Y_{2}=\{b,c\}$ with the quotas $%
q=2$, $q_{1}=q_{2}=1$ and the preference order $a\succ b\succ c$. Then for $%
A=\{b,c\}$ and $B=\{a,b,c\}$ we have $C(A)=\{b\}$ and $C(B)=\{a,c\}$, so
that $A\subset B$ but $A\cap C(B)\not\subset C(A)$.
\end{example}

Theorem \ref{t310} can be often used for the construction of \emph{individual%
} revealing choice functions. The following result enables us to combine
individual revealing choice functions into \emph{global} revealing choice
functions.

\begin{theorem}
\label{t313} Given a set function $C:2^X\to 2^X$ and a partition $X=\cup X_i$
with disjoint sets $X_i$, we define the set functions $C_i:2^{X_i}\to 2^{X_i}
$ by the formula $C_i(A_i):=C(A_i)\cap X_i$. Then $C$ is a revealing choice
map on $X$ if and only if each $C_i$ is a revealing choice map on $X_i$.
\end{theorem}

\section{Solution to the simple job market problem}

\label{s4}

For the solution we set 
\begin{equation*}
W:=\{w_1, w_2, w_3\},\quad F:=\{f_1, f_2, f_3\},\quad D:=\{d_1,d_2\}
\end{equation*}
and we proceed in several steps. \medskip

\emph{Step 1.} For each fixed worker $w_i$ we define a revealing choice map $%
C_{w_i}$ on $\{w_i\}\times F\times D$ by applying Theorem \ref{t310} with $Y$%
, $q$, $Y_n$ and $q_n$ given below. For brevity we write $(i,j,k)$ instead
of $(w_i,f_j,d_k)$ in the preference relations.

\begin{itemize}
\item For worker $w_1$ we choose 
\begin{align*}
&Y:=\{w_1\}\times \{f_1,f_2,f_3\}\times \{d_1,d_2\}, \\
&Y_1:=\{w_1\}\times \{f_1,f_2,f_3\}\times \{d_1\}, \\
&Y_2:=\{w_1\}\times \{f_1,f_2,f_3\}\times \{d_2\}
\end{align*}
with quotas $q=6$ (which is ineffective), $q_1=q_2=1$ and the following
preference relation on $Y$: 
\begin{equation*}
(1,1,1)\succ (1,1,2)\succ (1,2,1)\succ (1,2,2)\succ (1,3,1)\succ (1,3,2).
\end{equation*}

\item For worker $w_2$ we choose 
\begin{align*}
&Y:=\{w_2\}\times \{f_1,f_2,f_3\}\times \{d_1,d_2\}, \\
&Y_1:=\{w_2\}\times \{f_1,f_2,f_3\}\times \{d_1\}, \\
&Y_2:=\{w_2\}\times \{f_1,f_2,f_3\}\times \{d_2\}
\end{align*}
with quotas $q=6$, $q_1=q_2=1$ and the following preference relation on $Y$: 
\begin{equation*}
(2,2,1)\succ (2,2,2)\succ (2,3,1)\succ (2,3,2)\succ (2,1,1)\succ (2,1,2).
\end{equation*}

\item For worker $w_3$ we choose 
\begin{align*}
&Y:=\{w_3\}\times \{f_1,f_2,f_3\}\times \{d_1,d_2\}, \\
&Y_1:=\{w_3\}\times \{f_1,f_2,f_3\}\times \{d_1\}, \\
&Y_2:=\{w_3\}\times \{f_1,f_2,f_3\}\times \{d_2\}
\end{align*}
with quotas $q=6$, $q_1=q_2=1$ and the following preference relation on $Y$: 
\begin{equation*}
(3,3,1)\succ (3,3,2)\succ (3,1,1)\succ (3,1,2)\succ (3,2,1)\succ (3,2,2).
\end{equation*}
\end{itemize}

\medskip

\emph{Step 2.} Applying Theorem \ref{t313} we combine the three choice maps
of the preceding step into a global revealing choice map $C_{W}$ on $W\times
F\times D$ by setting 
\begin{equation*}
C_{W}(A):=\bigcup_{i=1}^{3}C_{w_{i}}\left( A\cap (\{w_{i}\}\times F\times
D)\right) 
\end{equation*}%
for every $A\subset W\times F\times D$.

\bigskip 

\emph{Step 3.} For each firm $f_{j}$ we define a revealing choice map $%
C_{f_{j}}$ on $W\times \{f_{j}\}\times D$ by applying Theorem \ref{t310}
with $Y$, $q$, $Y_{n}$ and $q_{n}$ given below and still writing $(i,j,k)$
instead of $(w_{i},f_{j},d_{k})$ for brevity.

\begin{itemize}
\item For firm $f_1$ we choose 
\begin{align*}
&Y:=\{w_1, w_2, w_3\}\times \{f_1\}\times \{d_1,d_2\}, \\
&Y_1:=\{w_1\}\times \{f_1\}\times \{d_1,d_2\}, \\
&Y_2:=\{w_2\}\times \{f_1\}\times\{d_1,d_2\}, \\
&Y_3:=\{w_3\}\times \{f_1\}\times\{d_1,d_2\},
\end{align*}
with quotas $q=2$, $q_1=q_2=q_3=1$ and the following preference relation on $%
Y$: 
\begin{equation*}
(2,1,2)\succ (2,1,1)\succ (3,2,2)\succ (3,2,1)\succ (1,3,2)\succ (1,3,1).
\end{equation*}

\item For firm $f_2$ we choose 
\begin{align*}
&Y:=\{w_1, w_2, w_3\}\times \{f_2\}\times \{d_1,d_2\}, \\
&Y_1:=\{w_1\}\times \{f_2\}\times \{d_1,d_2\}, \\
&Y_2:=\{w_2\}\times \{f_2\}\times\{d_1,d_2\}, \\
&Y_3:=\{w_3\}\times \{f_2\}\times\{d_1,d_2\},
\end{align*}
with quotas $q=2$, $q_1=q_2=q_3=1$ and the following preference relation on $%
Y$: 
\begin{equation*}
(3,2,2)\succ (3,2,1)\succ (1,2,2)\succ (1,2,1)\succ (2,2,2)\succ (2,2,1).
\end{equation*}

\item For firm $f_3$ we choose 
\begin{align*}
&Y:=\{w_1, w_2, w_3\}\times \{f_3\}\times \{d_1,d_2\}, \\
&Y_1:=\{w_1\}\times \{f_3\}\times \{d_1,d_2\}, \\
&Y_2:=\{w_2\}\times \{f_3\}\times\{d_1,d_2\}, \\
&Y_3:=\{w_3\}\times \{f_3\}\times\{d_1,d_2\},
\end{align*}
with quotas $q=2$, $q_1=q_2=q_3=1$ and the following preference relation on $%
Y$: 
\begin{equation*}
(1,3,2)\succ (1,3,1)\succ (2,3,2)\succ (2,3,1)\succ (3,3,2)\succ (3,3,1).
\end{equation*}
\end{itemize}

\medskip

\emph{Step 4.} Applying Theorem \ref{t313} we combine the three choice maps
of the preceding step into a global revealing choice map $C_F$ on $W\times
F\times D$ by setting 
\begin{equation*}
C_F(A):= \bigcup_{j=1}^3C_{f_j}\left( A\cap (W\times \{f_j\}\times D)\right)
,\quad A\subset W\times F\times D.
\end{equation*}

\medskip

\emph{Step 5.} The choice maps $C_W$ and $C_F$ satisfy the hypotheses of
Theorem \ref{t37}. We apply the algorithm as described in Remark \ref{r38}
(b) by starting with $X_0:=X$ and computing $Y_1, X_2, Y_3,X_4$ by the
formulae 
\begin{equation*}
Y_{n+1}:=(X\setminus X_n)\cup C_W(X_n) \quad\text{and}\quad
X_{n+1}:=(X\setminus Y_n)\cup C_F(Y_n).
\end{equation*}
We obtain that $X_2=X_4$ and therefore $S=C_W(X_2)$. The results are
summarized in the following table:

\begin{center}
\begin{tabular}{c|c|c||c|c|c|c|c|c}
$w_i$ & $f_j$ & $d_k$ & $X_0$ & $Y_1$ & $X_2$ & $Y_3$ & $X_4$ & $S$ \\ 
\hline\hline
1 & 1 & 1 & $x$ & $x$ &  & $x$ &  &  \\ 
1 & 1 & 2 & $x$ & $x$ & $x$ & $x$ & $x$ & $x$ \\ 
1 & 2 & 1 & $x$ &  & $x$ & $x$ & $x$ & $x$ \\ 
1 & 2 & 2 & $x$ &  & $x$ &  & $x$ &  \\ 
1 & 3 & 1 & $x$ &  & $x$ &  & $x$ &  \\ 
1 & 3 & 2 & $x$ &  & $x$ &  & $x$ &  \\ \hline
2 & 1 & 1 & $x$ &  & $x$ &  & $x$ &  \\ 
2 & 1 & 2 & $x$ &  & $x$ &  & $x$ &  \\ 
2 & 2 & 1 & $x$ & $x$ &  & $x$ &  &  \\ 
2 & 2 & 2 & $x$ & $x$ & $x$ & $x$ & $x$ & $x$ \\ 
2 & 3 & 1 & $x$ &  & $x$ & $x$ & $x$ & $x$ \\ 
2 & 3 & 2 & $x$ &  & $x$ &  & $x$ &  \\ \hline
3 & 1 & 1 & $x$ &  & $x$ & $x$ & $x$ & $x$ \\ 
3 & 1 & 2 & $x$ &  & $x$ &  & $x$ &  \\ 
3 & 2 & 1 & $x$ &  & $x$ &  & $x$ &  \\ 
3 & 2 & 2 & $x$ &  & $x$ &  & $x$ &  \\ 
3 & 3 & 1 & $x$ & $x$ &  & $x$ &  &  \\ 
3 & 3 & 2 & $x$ & $x$ & $x$ & $x$ & $x$ & $x$%
\end{tabular}
\end{center}

In this \emph{worker-optimal} solution each worker is hired by the second
most preferred firm for the first day and by the most preferred firm for the
second day. \medskip

\emph{Step 6.} Applying the algorithm of Remark \ref{r38} (b) by starting
with $Y_0:=X$ and computing $X_1, Y_2, X_3, Y_4$ by the above formulae we
obtain that $Y_2=Y_4$ and therefore $S=C_F(Y_2)$. The results are summarized
in the following table:

\begin{center}
\begin{tabular}{c|c|c||c|c|c|c|c|c}
$w_i$ & $f_j$ & $d_k$ & $Y_0$ & $X_1$ & $Y_2$ & $X_3$ & $Y_4$ & $S$ \\ 
\hline\hline
1 & 1 & 1 & $x$ &  & $x$ &  & $x$ &  \\ 
1 & 1 & 2 & $x$ &  & $x$ &  & $x$ &  \\ 
1 & 2 & 1 & $x$ &  & $x$ &  & $x$ &  \\ 
1 & 2 & 2 & $x$ & $x$ & $x$ & $x$ & $x$ & $x$ \\ 
1 & 3 & 1 & $x$ &  & $x$ & $x$ & $x$ & $x$ \\ 
1 & 3 & 2 & $x$ & $x$ &  & $x$ &  &  \\ \hline
2 & 1 & 1 & $x$ &  & $x$ & $x$ & $x$ & $x$ \\ 
2 & 1 & 2 & $x$ & $x$ &  & $x$ &  &  \\ 
2 & 2 & 1 & $x$ &  & $x$ &  & $x$ &  \\ 
2 & 2 & 2 & $x$ &  & $x$ &  & $x$ &  \\ 
2 & 3 & 1 & $x$ &  & $x$ &  & $x$ &  \\ 
2 & 3 & 2 & $x$ & $x$ & $x$ & $x$ & $x$ & $x$ \\ \hline
3 & 1 & 1 & $x$ &  & $x$ &  & $x$ &  \\ 
3 & 1 & 2 & $x$ & $x$ & $x$ & $x$ & $x$ & $x$ \\ 
3 & 2 & 1 & $x$ &  & $x$ & $x$ & $x$ & $x$ \\ 
3 & 2 & 2 & $x$ & $x$ &  & $x$ &  &  \\ 
3 & 3 & 1 & $x$ &  & $x$ &  & $x$ &  \\ 
3 & 3 & 2 & $x$ &  & $x$ &  & $x$ & 
\end{tabular}
\end{center}

In this \emph{firm-optimal} solution each firm hires the most preferred
worker for the first day and by the second most preferred worker for the
second day.

\begin{remark}
\label{r41} The stable schedule matchings as studied by Baiou and Balinski 
\cite{BaiBal2002} and Alkan and Gale \cite{AlkGal2003} enter the present
framework as a special case. For simplicity we consider the discrete case
and we denote by $D=\{1,2,\ldots \}$ the possible number of working hours
with $k$ meaning the $k$th working hour. For each worker $w_{i}$, if there
is a preference ranking $f_{j_{1}}\succ f_{j_{2}}\succ \cdots $ among the
firms, then we extend it to the preference ranking 
\begin{align*}
& (i,j_{1},1)\succ (i,j_{1},2)\succ \cdots \succ (i,j_{1},q_{i,j_{1}}^{w}) \\
\succ & (i,j_{2},1)\succ (i,j_{2},2)\succ \cdots \succ
(i,j_{2},q_{i,j_{2}}^{w}) \\
\succ & \cdots  \\
& \vdots 
\end{align*}%
where $q_{i,j}^{w}$ denotes the maximum number of working hours accepted by
worker $w_{i}$ in firm $f_{j}$. Similarly, for each firm $f_{j}$, if there
is a preference ranking $w_{i_{1}}\succ w_{i_{2}}\succ \cdots $ among the
workers, then we extend it to the preference ranking 
\begin{align*}
& (i_{1},j,1)\succ (i_{1},j,2)\succ \cdots \succ (i_{1},j,q_{i_{1},j}^{f}) \\
\succ & (i_{2},j,1)\succ (i_{2},j,2)\succ \cdots \succ
(i_{2},j,q_{i_{2},j}^{f}) \\
\succ & \cdots  \\
& \vdots 
\end{align*}%
where $q_{i,j}^{f}$ denotes the maximum number of working hours accepted by
firm $f_{j}$ for worker $w_{i}$. Once a stable set $S$ found, the number of
working hours of worker $w_{i}$ in firm $f_{j}$ is the biggest integer $k$
such that $(i,j,k)\in S$.
\end{remark}

\section{More complex examples}

\label{s5}

We illustrate in this section the strength and flexibility of our theorems
and algorithms by solving some more complex problems.

We consider the following modeling issue. We are given a finite number of
workers $w_{i}$, firms $f_{j}$ and days $d_{k}$ (days of a week or days of a
month for instance). Each worker may work at one or several firms per day,
maybe at different firms on different days. Similarly, each firm may hire a
given number of workers per day, maybe differents numbers on different days.

\bigskip 

A \emph{contract} is by definition a triple $(w_{i},f_{j},d_{k})$ meaning
that worker $w_{i}$ is hired by firm $f_{j}$ for day $d_{k}$, and we are
looking for an acceptable set of contracts, subject to various requirements
of both workers and firms. Thus, each worker $w_{i}$

\begin{itemize}
\item may exclude some firm--day pairs $(f_j,d_k)$ considered unacceptable;

\item has a strict preference ordering among the remaining firm--day pairs;

\item may put some other restrictions, such as

\begin{itemize}
\item to set a maximum quota $q^w_i$ of accepted firm--day pairs;

\item not to work on day $d_k$ at more than a given number $q^w_{i,k}$ of
firms;

\item or not to work at firm $f_j$ more than a given number $\tilde q^w_{i,j}
$ of days.
\end{itemize}
\end{itemize}

Similarly, each firm $f_j$

\begin{itemize}
\item may exclude some worker--day pairs $(w_i,d_k)$ considered unacceptable;

\item has a strict preference ordering among the remaining worker--day pairs;

\item may put some other restrictions, such as

\begin{itemize}
\item to set a maximum quota $q^f_j$ of worker--day pairs for hiring;

\item not to hire on day $d_k$ more than a given number $q^f_{j,k}$ of
workers;

\item or not to hire worker $w_i$ for more than a given number $\tilde
q^f_{i,j}$ of days.
\end{itemize}
\end{itemize}

\begin{remarks}
\label{r51}\mbox{}

\begin{itemize}
\item[(a)] Although we keep strict preference ordering on worker-day pairs
or on firm-day pairs, these preference ordering are not sufficient to
characterize the choice map of each agent: they also depend on the quota
system.

\item[(b)] In most applications we may assume that a worker does not work at
more than one firm per day, so that $q^w_{i,k}=1$ for every $k$; then $q_i$
means the maximum number of working days for the worker $w_i$.
\end{itemize}
\end{remarks}

\subsection{First problem}

\label{ss51}

Assume that we have four workers $w_1$, $w_2$, $w_3$, $w_4$ and three firms $%
f_1$, $f_2$, $f_3$. Each worker may work at most at one firm per day, maybe
at different firms on different days of the week. The further requirements
of the agents are listed below.

\begin{itemize}
\item Worker $w_1$ can work at most 4 days per week, with the following
strict preference order of the firm--day pairs $(f_j,d_k)$ where we write $%
(j,k)$ instead of $(f_j,d_k)$ for brevity: 
\begin{align}  \label{51}
&(2,1)\succ (3,1)\succ (2,2)\succ (3,2)\succ (2,3)\succ (3,3)\succ (2,4) \\
\succ &(3,4)\succ (2,5)\succ (3,5)\succ (2,6)\succ (3,6)\succ (2,7)\succ
(3,7) .  \notag
\end{align}
This list shows for instance that worker $w_1$ prefers most to be hired by
firm $f_2$ for Mondays ($d_1$), then by firm $f_3$ always for Mondays, next
by firm $f_2$ for Tuesdays ($d_2$), and so on. The absence of firm $f_1$ in
the list shows that worker $w_1$ refuses to be hired by that firm.

\item Worker $w_2$ can work at most 3 days per week, with the following
strict preference order: 
\begin{align}  \label{52}
&(1,1)\succ (1,2)\succ (1,3)\succ (1,4)\succ (1,5)\succ (1,6)\succ (1,7) \\
\succ &(2,1)\succ (2,2)\succ (2,3)\succ (2,4)\succ (2,5)\succ (2,6)\succ
(2,7)  \notag \\
\succ &(3,1)\succ (3,2)\succ (3,3)\succ (3,4)\succ (3,5)\succ (3,6)\succ
(3,7).  \notag
\end{align}

\item Worker $w_3$ can work at most 2 days per week, with the following
strict preference order: 
\begin{align}  \label{53}
&(2,2)\succ (2,3)\succ (3,2)\succ (3,3)\succ (1,2)\succ (1,3) \\
\succ &(2,4)\succ (2,5)\succ (2,6)\succ (1,4)\succ (1,5)\succ (1,6)  \notag
\\
\succ &(3,6)\succ (3,4)\succ (3,5).  \notag
\end{align}
The list shows in particular that he/she does not work on Mondays ($d_1$)
and Sundays ($d_7$).

\item Worker $w_4$ accepts to work on all days of the week, with the
following strict preference order: 
\begin{align}  \label{54}
&(1,1)\succ (1,2)\succ (1,3)\succ (1,4)\succ (1,5)\succ (1,6)\succ (1,7) \\
\succ &(2,1)\succ (2,2)\succ (2,3)\succ (2,4)\succ (2,5)\succ (2,6)\succ
(2,7)  \notag \\
\succ &(3,1)\succ (3,2)\succ (3,3)\succ (3,4)\succ (3,5)\succ (3,6)\succ
(3,7).  \notag
\end{align}

\item Firm $f_1$ may hire up to 4 workers per day when it is open, with the
following strict preference order of the worker--day pairs $(w_i,d_k)$ where
we write $(i,k)$ instead of $(w_i,d_k)$ for brevity: 
\begin{align}  \label{55}
&(1,1)\succ (1,2)\succ (1,3)\succ (1,4)\succ (1,5)\succ (1,6) \\
\succ &(2,1)\succ (2,2)\succ (2,3)\succ (2,4)\succ (2,5)\succ (2,6)  \notag
\\
\succ &(3,1)\succ (3,2)\succ (3,3)\succ (3,4)\succ (3,5)\succ (3,6).  \notag
\end{align}
The list shows in particular that the firm is closed on Sundays and that it
doesn't hire worker $w_4$. Otherwise, it prefers most to hire worker $w_1$
for Mondays, then worker $w_1$ for Tuesdays, and so on.

\item Firm $f_2$ may also hire up to 4 workers per day, with the following
strict preference order: 
\begin{align}  \label{56}
&(3,7)\succ (3,6)\succ (3,5)\succ (3,4)\succ (3,3)\succ (3,2)\succ (3,1) \\
\succ &(4,7)\succ (4,6)\succ (4,5)\succ (4,4)\succ (4,3)\succ (4,2)\succ
(4,1)  \notag \\
\succ &(1,7)\succ (1,6)\succ (1,5)\succ (1,4)\succ (1,3)\succ (1,2)\succ
(1,1)  \notag \\
\succ &(2,7)\succ (2,6)\succ (2,5)\succ (2,4)\succ (2,3)\succ (2,2)\succ
(2,1).  \notag
\end{align}

\item Firm $f_3$ is closed on Saturdays and Sundays; for the other days it
may hire up to 4 workers per day, with the following strict preference
order: 
\begin{align}  \label{57}
&(4,1)\succ (4,2)\succ (4,3)\succ (4,4)\succ (4,5) \\
\succ &(3,1)\succ (3,2)\succ (3,3)\succ (3,4)\succ (3,5)  \notag \\
\succ &(2,1)\succ (2,2)\succ (2,3)\succ (2,4)\succ (2,5)  \notag \\
\succ &(1,1)\succ (1,2)\succ (1,3)\succ (1,4)\succ (1,5).  \notag
\end{align}
\end{itemize}

Our task is to find an acceptable firm-worker assignment and work schedule
under these constraints. For the solution we set 
\begin{equation*}
W:=\{w_1, w_2, w_3, w_4\},\quad F:=\{f_1, f_2, f_3\},\quad
D:=\{d_1,\ldots,d_7\}
\end{equation*}
and we proceed in several steps. \medskip

\emph{Step 1.} For each fixed worker $w_i$ we define a revealing choice map $%
C_{w_i}$ on $\{w_i\}\times F\times D$ by applying Theorem \ref{t310} with $%
Y:=Y^w_i$, $q:=q^w_i$, $Y_n:=Y^w_{i,n}$ and $q_n:=q^w_{i,n}$ given below.
For brevity we write $(i,j,k)$ instead of $(w_i,f_j,d_k)$ in the preference
relations.

\begin{itemize}
\item For worker $w_1$ we choose 
\begin{equation*}
Y^w_1:=\{w_1\}\times \{f_2,f_3\}\times D
\end{equation*}
representing the set of acceptable firms and days of worker $w_1$, with
quota $q^w_1=4$ and the following preference relation on $Y^w_1$ (see %
\eqref{51}): 
\begin{align*}
&(1,2,1)\succ (1,3,1)\succ (1,2,2)\succ (1,3,2)\succ (1,2,3) \\
\succ &(1,3,3)\succ (1,2,4)\succ (1,3,4)\succ (1,2,5)\succ (1,3,5) \\
\succ &(1,2,6)\succ (1,3,6)\succ (1,2,7)\succ (1,3,7) .
\end{align*}
Furthermore, we set 
\begin{equation*}
Y^w_{1,k}:=\{w_1\}\times F\times \{d_k\}\text{ and } q^w_{1,k}=1\text{ for }
k=1,\ldots, 7.
\end{equation*}

\item For worker $w_2$ we choose 
\begin{equation*}
Y^w_2:=\{w_2\}\times F\times D
\end{equation*}
with quota $q^w_2=3$ and the preference relation 
\begin{align*}
&(2,1,1)\succ (2,1,2)\succ (2,1,3)\succ (2,1,4)\succ (2,1,5) \\
\succ &(2,1,6)\succ (2,1,7)\succ (2,2,1)\succ (2,2,2)\succ (2,2,3) \\
\succ &(2,2,4)\succ (2,2,5)\succ (2,2,6)\succ (2,2,7)\succ (2,3,1) \\
\succ &(2,3,2)\succ (2,3,3)\succ (2,3,4)\succ (2,3,5)\succ (2,3,6)\succ
(2,3,7).
\end{align*}
on $Y^w_2$ (see \eqref{52}). Furthermore, we set 
\begin{equation*}
Y^w_{2,k}:=\{w_2\}\times F\times \{d_k\}\text{ and } q^w_{2,k}=1\text{ for }
k=1,\ldots, 7.
\end{equation*}

\item For worker $w_3$ we choose 
\begin{equation*}
Y^w_3:=\{w_3\}\times F\times \{d_2,\ldots,d_6\}
\end{equation*}
with quota $q^w_3=2$ and the preference relation 
\begin{align*}
&(3,2,2)\succ (3,2,3)\succ (3,3,2)\succ (3,3,3)\succ (3,1,2) \\
\succ &(3,1,3)\succ (3,2,4)\succ (3,2,5)\succ (3,2,6)\succ (3,1,4) \\
\succ &(3,1,5)\succ (3,1,6)\succ (3,3,6)\succ (3,3,4)\succ (3,3,5).
\end{align*}
on $Y^w_3$ (see \eqref{53}). Furthermore, we set 
\begin{equation*}
Y^w_{3,k}:=\{w_3\}\times F\times \{d_k\}\text{ and } q^w_{3,k}=1\text{ for }
k=1,\ldots, 7.
\end{equation*}

\item For worker $w_4$ we choose 
\begin{equation*}
Y^w_4:=\{w_4\}\times F\times D
\end{equation*}
with quota $q^w_4=7$ and the preference relation 
\begin{align*}
&(4,1,1)\succ (4,1,2)\succ (4,1,3)\succ (4,1,4)\succ (4,1,5) \\
\succ &(4,1,6)\succ (4,1,7)\succ (4,2,1)\succ (4,2,2)\succ (4,2,3) \\
\succ &(4,2,4)\succ (4,2,5)\succ (4,2,6)\succ (4,2,7)\succ (4,3,1) \\
\succ &(4,3,2)\succ (4,3,3)\succ (4,3,4)\succ (4,3,5)\succ (4,3,6)\succ
(4,3,7).
\end{align*}
on $Y^w_4$ (see \eqref{54}). Furthermore, we set 
\begin{equation*}
Y^w_{4,k}:=\{w_4\}\times F\times \{d_k\}\text{ and } q^w_{4,k}=1\text{ for }
k=1,\ldots, 7.
\end{equation*}
\end{itemize}

\medskip

\emph{Step 2.} Applying Theorem \ref{t313} we combine the four choice maps
of the preceding step into a global revealing choice map $C_{W}$ on $W\times
F\times D$ by setting 
\begin{equation*}
C_{W}(A):=\bigcup_{i=1}^{4}C_{w_{i}}\left( A\cap (\{w_{i}\}\times F\times
D)\right) 
\end{equation*}%
for every $A\subset W\times F\times D$.

\bigskip 

\emph{Step 3.} For each firm $f_{j}$ we define a revealing choice map $%
C_{f_{j}}$ on $W\times \{f_{j}\}\times D$ by applying Theorem \ref{t310}
again, this time with $Y:=Y_{j}^{f}$, $q:=q_{j}^{f}$, $Y_{n}:=Y_{j,n}^{f}$
and $q_{n}:=q_{j,n}^{f}$ given below and still writing $(i,j,k)$ instead of $%
(w_{i},f_{j},d_{k})$ for brevity.

\begin{itemize}
\item For firm $f_1$ we choose 
\begin{equation*}
Y^f_1:=W\times \{f_1\}\times \{d_1,\ldots,d_6\}
\end{equation*}
with quota $q^f_1=24$ and the following preference relation on $Y^f_1$ (see %
\eqref{55}): 
\begin{align*}
&(1,1,1)\succ (1,1,2)\succ (1,1,3)\succ (1,1,4)\succ (1,1,5) \\
\succ &(1,1,6)\succ (2,1,1)\succ (2,1,2)\succ (2,1,3)\succ (2,1,4) \\
\succ &(2,1,5)\succ (2,1,6)\succ (3,1,1)\succ (3,1,2)\succ (3,1,3) \\
\succ &(3,1,4)\succ (3,1,5)\succ (3,1,6).
\end{align*}

\item For firm $f_2$ we choose 
\begin{equation*}
Y^f_2:=W\times \{f_2\}\times D
\end{equation*}
with quota $q^f_2=28$ and the following preference relation on $Y^f_2$ (see %
\eqref{57}): 
\begin{align*}
&(3,2,7)\succ (3,2,6)\succ (3,2,5)\succ (3,2,4)\succ (3,2,3) \\
\succ &(3,2,2)\succ (3,2,1)\succ (4,2,7)\succ (4,2,6)\succ (4,2,5) \\
\succ &(4,2,4)\succ (4,2,3)\succ (4,2,2)\succ (4,2,1)\succ (1,2,7) \\
\succ &(1,2,6)\succ (1,2,5)\succ (1,2,4)\succ (1,2,3)\succ (1,2,2) \\
\succ &(1,2,1)\succ (2,2,7)\succ (2,2,6)\succ (2,2,5)\succ (2,2,4) \\
\succ &(2,2,3)\succ (2,2,2)\succ (2,2,1).
\end{align*}

\item For firm $f_3$ we choose 
\begin{equation*}
Y^f_3:=W\times \{f_3\}\times \{d_1,\ldots,d_5\}
\end{equation*}
with quota $q^f_3=20$ and the following preference relation on $Y^f_3$ (see %
\eqref{57}): 
\begin{align*}
&(4,3,1)\succ (4,3,2)\succ (4,3,3)\succ (4,3,4)\succ (4,3,5) \\
\succ &(3,3,1)\succ (3,3,2)\succ (3,3,3)\succ (3,3,4)\succ (3,3,5) \\
\succ &(2,3,1)\succ (2,3,2)\succ (2,3,3)\succ (2,3,4)\succ (2,3,5) \\
\succ &(1,3,1)\succ (1,3,2)\succ (1,3,3)\succ (1,3,4)\succ (1,3,5).
\end{align*}
\end{itemize}

\medskip

\emph{Step 4.} Applying Theorem \ref{t313} we combine the three choice maps
of the preceding step into a global revealing choice map $C_F$ on $W\times
F\times D$ by setting 
\begin{equation*}
C_F(A):= \bigcup_{j=1}^3C_{f_j}\left( A\cap (W\times \{f_j\}\times D)\right)
,\quad A\subset W\times F\times D.
\end{equation*}

\medskip

\emph{Step 5.} The choice maps $C_{W}$ and $C_{F}$ satisfy the hypotheses of
Theorem \ref{t37}. Applying the algorithm as described in Remark \ref{r38}
(b) by starting with $X_{0}:=X$, we use a computer program to make the
otherwise tedious computation. We obtain the following worker-optimal stable
schedule: 
\begin{equation*}
(w_{1},f_{2},1-4),\quad (w_{2},f_{1},1-3),\quad (w_{3},f_{2},2-3),\quad
(w_{4},f_{2},1-7).
\end{equation*}%
The notations means that

\begin{itemize}
\item $f_1$ hires worker $w_2$ for Mondays, Tuesdays and Wednesdays;

\item $f_2$ hires worker $w_1$ for Mondays, Tuesdays, Wednesdays and
Thursdays, worker $w_3$ for Tuesdays and Wednesdays, and worker $w_4$ for
all seven days of the week;

\item $f_3$ does not hire anybody.
\end{itemize}

\medskip

\emph{Step 6.} Applying the algorithm of Remark \ref{r38} (b) by starting
with $Y_{0}:=X$ we obtain the same solution. This means that the
worker-optimal and firm-optimal solutions coincide, and that there is a
unique stable schedule matching in this case.

The remaining of this section investigates the changes in the solutions if
we modify our requirements in various ways.

\subsection{Second problem}

\label{ss52}

If worker $w_3$ accepts to work up to four days per week (so we change $%
q^w_3=2$ to $q^w_3=4$), then the worker-optimal and firm-optimal solutions
still coincide: the stable schedule is given by the list 
\begin{equation*}
(w_1,f_2,1-4),\quad (w_2,f_1,1-3),\quad (w_3,f_2,2-5),\quad (w_4,f_2,1-7).
\end{equation*}
The only change with respect to the preceding case is that $f_2$ now hires $%
w_3$ for Thursdays and Fridays, too.

\subsection{Third problem}

\label{ss53}

We modify the problem such that $f_2$ hires at most one worker per day, so
that for the construction of the choice map $C_{f_2}$ we add the extra
conditions 
\begin{equation*}
Y^f_{1,k}:=W\times \{f_2\}\times \{d_k\}\text{ and } q^f_{1,k}=1\text{ for }
k=1,\ldots, 7.
\end{equation*}
The changes are more important. Both the worker-optimal solution and
firm-optimal solutions are given by the list 
\begin{equation*}
(w_1,f_3,1-4),\quad (w_2,f_1,1-3),\quad (w_3,f_2,2-3),\quad (w_4,f_2,1,
4-7),\quad (w_4,f_3,2-3).
\end{equation*}

\subsection{Fourth problem}

\label{ss54}

Now assume that

\begin{itemize}
\item firm $f_1$ does not hire any worker for more than two days;

\item firm $f_2$ does not hire Worker $w_1$ for more than three days;

\item firm $f_2$ does not hire Worker $w_4$ for more than three days either.
\end{itemize}

We proceed as in Subsection \ref{ss51} but in constructing $C_{1}^{f}$ we
add the extra conditions 
\begin{equation*}
\tilde{Y}_{i,1}^{f}:=\{w_{i}\}\times \{f_{1}\}\times D\text{ and }\tilde{q}%
_{i,1}^{f}:=2\text{ for }i=1,2,3,4,
\end{equation*}%
and in constructing $C_{2}^{f}$ we add the extra conditions 
\begin{align*}
& \tilde{Y}_{1,2}^{f}:=\{w_{1}\}\times \{f_{2}\}\times D\text{ and }\tilde{q}%
_{1,2}^{f}:=3, \\
& \tilde{Y}_{4,2}^{f}:=\{w_{4}\}\times \{f_{2}\}\times D\text{ and }\tilde{q}%
_{1,2}^{f}:=3.
\end{align*}%
Now the worker-optimal and firm-optimal stable schedules differ: they are
given by 
\begin{align*}
& (w_{1},f_{2},2-4),\quad (w_{1},f_{3},1),\quad (w_{2},f_{1},1-2),\quad
(w_{2},f_{2},3), \\
& (w_{3},f_{2},2-3),\quad (w_{4},f_{2},5-7),\quad (w_{4},f_{3},1-4)
\end{align*}%
and 
\begin{align*}
& (w_{1},f_{2},5-7),\quad (w_{1},f_{3},1),\quad (w_{2},f_{1},1-2),\quad
(w_{2},f_{2},3), \\
& (w_{3},f_{2},2-3),\quad (w_{4},f_{2},5-7),\quad (w_{4},f_{3},1-4),
\end{align*}%
respectively.

\bigskip 

\section{Concluding remarks}

\label{s6}

The \emph{schedule matching} problem extends the standard matching procedure
to the allocation of real numbers (days, hours or quantities) between two
separate sets of agents. The present paper generalizes the notion of
schedule matching to allow for schedule and preference constraints on each
side of the market. We demonstrate, by several example, that the revealing
property of the choice maps is the most suitable one to ensure the existence
of stable matchings. We also revisit the mathematical structure of the
matching theory by comparing various definitions of stable sets and various
classes of choice maps.

The generality of our analysis \ is not only theoretically interesting but
is potentially useful in application as well. 

In certain, highly-competitive, labor markets employers perceive a shortage
of top-level candidates that lead to hiring strategies intended to hire
those who are believed to be the best. Competition within these paradigms
inevitably leads to ever-evolving, if not escalating, dynamic reactions on
both sides in an effort to maximize overall gain. As intended, the strategy
forces the candidates make ever quicker decisions, before they can know, and
weigh, other offers that may be proffered in the near future. As a
consequence, candidates end up having less opportunities and employers less
potential candidates than were originally available in the market. This
results in sub-optimal matches that spawn a myriad of both observable and
hidden costs on both sides of the market. The eventual failure of this
common strategy ultimately mandates sets of new rules and procedures or
market re-design. The algoritm proposed here could be used as a
``clearinghouse'' in situations where quotas are placed by workers on firms
and days worked, allowing him to work part-time for different firms on the
same day or on different days, excluding some firms on some given days or
excluding some days of work. In the same manner, the algorithm is applicable
to situations where firms to need to adjust their labor force on certain
days depending on their anticipated activity, or on the requirements
associated to different activities on different days or the same day. 

\section{Proof of the theorems of Section \protect\ref{s3} and supplementary
results}

\label{s7}

First we prove Propositions \ref{p35} and \ref{p36}. Then we apply them to
establish Theorem \ref{t37}. In the second, independent part of the section
we prove Theorems \ref{t310} and \ref{t313}.

\begin{proof}[Proof of Proposition \protect\ref{p35}]
\mbox{}

(a) Assume that $C:2^X\to 2^X$ is revealing. Then it is persistent because $%
C(A)\subset A$ for every choice map.

In order to prove the consistence first we observe that in case $C(A)\subset
B$ we infer from our hypothesis $C(A)\subset B\subset A$ and from the choice
map property $C(B)\subset B$ that $C(B)\subset A$. Therefore using \eqref{37}
we have 
\begin{align*}
\intertext{and}
&C(B)\subset A \Longrightarrow B\cap C(A)\subset C(B) \Longleftrightarrow
C(A)\subset C(B),
\end{align*}
so that $C(A)= C(B)$. \medskip

Now assume that $C:2^X\to 2^X$ is consistent and persistent, and consider
two sets satisfying $C(A)\subset B$. We have to prove that $A\cap
C(B)\subset C(A)$.

Since $A\subset A\cup B$, applying \eqref{39} we obtain that 
\begin{equation}  \label{61}
A\cap C(A\cup B)\subset C(A).
\end{equation}
The proof will be completed by showing that $C(A\cup B)=C(B)$.

Using the hypothesis $C(A)\subset B$ we deduce from \eqref{61} that $C(A\cup
B)\subset B$. Therefore $C(A\cup B)\subset B\subset A\cup B$, and the
equality $C(A\cup B)=C(B)$ follows by applying \eqref{38}. \medskip

(b) If the choice map is consistent, then its idempotence follows by
applying \eqref{38} with $B=C(A)$. If $C$ is persistent, then applying %
\eqref{39} with $A=C(B)$ we get $C(B)\subset C(C(B))$. The converse
inclusion also holds because $C$ is a choice map. \medskip

(c) If $A\subset B$, then 
\begin{align*}
R(A)\subset R(B)& \Longleftrightarrow A\setminus C(A)\subset B\setminus C(B)
\\
& \Longleftrightarrow A\setminus C(A)\subset A\setminus C(B) \\
& \Longleftrightarrow A\cap C(B)\subset C(A).\qedhere
\end{align*}
\end{proof}

\begin{examples}
\label{e61}\mbox{}

(a) Consider a two-point set $X=\{a,b\}$ and the four choice maps defined by
the following formulae:

\begin{tabular}{c||c|c|c|c}
$A$ & $\varnothing$ & $\{a\}$ & $\{b\}$ & $\{a,b\}$ \\ \hline\hline
$C_1(A)$ & $\varnothing$ & $\{a\}$ & $\{b\}$ & $\{a\}$ \\ \hline
$C_2(A)$ & $\varnothing$ & $\varnothing$ & $\{b\}$ & $\varnothing$ \\ \hline
$C_3(A)$ & $\varnothing$ & $\varnothing$ & $\varnothing$ & $\{a,b\}$ \\ 
\hline
$C_4(A)$ & $\varnothing$ & $\varnothing$ & $\varnothing$ & $\{a\}$%
\end{tabular}

One may readily verify that

\begin{itemize}
\item $C_1$ is revealing,

\item $C_2$ is persistent but not consistent,

\item $C_3$ is consistent but not persistent,

\item $C_4$ is not idempotent.
\end{itemize}

One may check that every idempotent choice map on $X$ is either consistent
or persistent (or both). \medskip

(b) Consider a three-point set $X=\{a,b,c\}$ and the choice map $%
C_{5}:2^{X}\rightarrow 2^{X}$ defined by

$%
\begin{array}{ccc}
C_{5}(\{a\})=\varnothing , & C_{5}(X)=\{b\}, & C_{5}(A)=A\quad \text{%
otherwise.}%
\end{array}%
$ 

Then $C_{5}$ is idempotent but neither consistent, nor persistent.\bigskip 
\end{examples}

\begin{proof}[Proof of Proposition \protect\ref{p36}]
\mbox{}

(a) Using the idempotence of $C_{W}$ and $C_{F}$ we deduce from \eqref{312}
that 
\begin{equation*}
S=C_{F}(S_{F})=C_{F}(C_{F}(S))=C_{F}(S).
\end{equation*}

(b) Assume that $C_{F}$ is consistent (the other case is similar) set $%
S_{F}^{\prime }:=S\cup (X\setminus S_{W})$. Then $S_{W}\cup S_{F}^{\prime }=X
$, $S_{W}\cap S_{F}^{\prime }=S$ and we still have $C_{W}(S_{W})=S$.
Furthermore, since 
\begin{equation*}
C_{F}(S_{F})=S\subset S_{F}^{\prime }\subset S_{F},
\end{equation*}%
using the consistence of $C_{F}$ we conclude that $C_{F}(S_{F}^{\prime })=S$%
. \medskip 

(c) As we already observed \eqref{32}--\eqref{34} imply \eqref{312}.
Conversely, \eqref{312} contains \eqref{32}; furthermore, by the consistency
of $C_{W}$ and $C_{F},$ \eqref{312} implies \eqref{33}--\eqref{34}. \medskip 

(d) If $S$ is a stable set, then \eqref{35}--\eqref{36} follow from %
\eqref{34}--\eqref{37}. Now assume \eqref{35} and \eqref{36}. Setting 
\begin{equation*}
S_{W}:=\{x\in X\ :\ C_{W}(S\cup \{x\})=S\}\quad \text{and}\quad
S_{F}:=\{x\in X\ :\ C_{F}(S\cup \{x\})=S\}
\end{equation*}%
we have $S_{W}\cup S_{F}=X$ by \eqref{36}. In view of the consistence it
remains to show that $C_{W}(S_{W})=S=C_{F}(S_{F})$.

If $x\in S_{W}\setminus S$, then applying the revealed preference property
and using \eqref{36} we deduce from the inclusion $C_{W}(S\cup \{x\})\subset
S_{W}$ that 
\begin{equation*}
(S\cup \{x\})\cap C_{W}(S_{W})\subset C_{W}(S\cup \{x\})=S
\end{equation*}%
and hence $x\notin C_{W}(S_{W})$. We have thus $C_{W}(S_{W})\subset S$.
Applying again the revealed preference property we deduce from this last
inclusion that 
\begin{equation*}
S_{W}\cap C_{W}(S)\subset C_{W}(S_{W}).
\end{equation*}%
Since $C_{W}(S)=S$ by \eqref{35}, it follows that $S\subset C_{W}(S_{W})$,
so that finally $C_{W}(S_{W})=S$. The proof of $C_{F}(S_{F})=S$ is similar.
\end{proof}

\begin{proof}[Proof of Theorem \protect\ref{t37}]
Let us introduce the map $f:2^{X}\times 2^{X}\rightarrow 2^{X}\times 2^{X}$
by the formula 
\begin{equation*}
f(A,B):=(X\setminus R_{F}(B),X\setminus R_{W}(A))
\end{equation*}%
where $R_{F}$, $R_{W}$ denote the rejection maps corresponding to $C_{F}$
and $C_{W}$. We observe that $2^{X}\times 2^{X}$ is a non-empty complete
lattice with respect to the order relation 
\begin{equation*}
(A,B)\leq (A^{\prime },B^{\prime })\Longleftrightarrow A\subset A^{\prime
}\quad \text{and}\quad B\supset B^{\prime }.
\end{equation*}%
Furthermore, $f$ is monotone with respect to this order relation. Indeed,
using the monotonicity of the rejection maps we have 
\begin{align*}
(A,B)\leq (A^{\prime },B^{\prime })& \Longleftrightarrow A\subset A^{\prime
}\quad \text{and}\quad B\supset B^{\prime } \\
& \Longrightarrow R_{W}(A)\subset R_{W}(A^{\prime })\quad \text{and}\quad
R_{F}(B)\supset R_{F}(B^{\prime }) \\
& \Longrightarrow X\setminus R_{F}(B)\subset X\setminus R_{F}(B^{\prime })%
\text{ and }X\setminus R_{W}(A)\supset X\setminus R_{W}(A^{\prime }) \\
& \Longleftrightarrow f(A,B)\leq f(A^{\prime },B^{\prime }).
\end{align*}%
Applying a fixed point theorem of Knaster and Tarski \cite{Kna1928}, \cite%
{Tar1928}, \cite{Tar1955} we conclude that $f$ has at least one fixed point
and that the fixed points of $f$ form a complete lattice. It remains to show
that the fixed points of $f$ coincide with the stable sets. More precisely,
in view of Proposition \ref{p36} it is sufficient to prove that 
\begin{equation*}
f(A,B)=(A,B)\Longleftrightarrow A\cup B=X\quad \text{and}\quad
C_{W}(A)=A\cap B=C_{F}(B).
\end{equation*}

If $f(A,B)=(A,B)$, then $A=X\setminus R_{F}(B)$ and $B=X\setminus R_{W}(A)$.
Since $R_{F}(B)\subset B$, it follows from the first relation that $A\cup B=X
$. Furthermore, the first relation also implies that $A$ is the disjoint
union of the sets $X\setminus B$ and $C_{F}(B)$ and hence that $A\cap
B\subset C_{F}(B)\subset A$. Since $C_{F}$ is a choice map, we also have $%
C_{F}(B)\subset B$ and therefore $C_{F}(B)=A\cap B$. The proof of the
equality $C_{W}(A)=A\cap B$ is analogous.

Conversely, if $A\cup B=X$ and $C_{W}(A)=A\cap B=C_{F}(B)$, then 
\begin{multline*}
X\setminus R_{F}(B)=(X\setminus B)\cup C_{F}(B) \\
=((A\cup B)\setminus B)\cup (A\cap B)=(A\setminus B)\cup (A\cap B)=A
\end{multline*}%
and 
\begin{multline*}
X\setminus R_{W}(A)=(X\setminus A)\cup C_{W}(A) \\
=((A\cup B)\setminus A)\cup (A\cap B)=(B\setminus A)\cup (A\cap B)=B,
\end{multline*}%
so that $f(A,B)=(A,B)$.
\end{proof}

\begin{remark}
\label{r62}\mbox{}

(a) In case $X$ is a finite set, the proof of the theorem provides an
efficient algorithm to find a stable set. Starting with $(X_{0},Y_{0}):=X%
\times \varnothing $ we define a sequence $(X_{1},Y_{1})$, $(X_{2},Y_{2})$%
,\ldots by the recursive relations 
\begin{equation*}
(X_{n+1},Y_{n+1}):=(X\setminus R_{F}(Y_{n}),X\setminus R_{W}(X_{n})),\quad
n=0,1,\ldots ,
\end{equation*}%
i.e., 
\begin{equation*}
X_{n+1}:=(X\setminus Y_{n})\cup C_{F}(Y_{n})\quad \text{and}\quad
Y_{n+1}:=(X\setminus X_{n})\cup C_{W}(X_{n}),\quad n=0,1,\ldots .
\end{equation*}%
Since we have obviously $X_{1}\subset X=X_{0}$ and $Y_{1}\supset \varnothing
=Y_{0}$, by the monotonicity of $f$ we conclude that 
\begin{equation*}
X_{0}\supset X_{1}\supset \cdots \quad \text{and}\quad Y_{0}\subset
Y_{1}\subset \cdots .
\end{equation*}%
Since $X$ has only finitely many subsets, there exists an index $n$ such
that 
\begin{equation*}
(X_{n+1},Y_{n+1})=(X_{n},Y_{n}),
\end{equation*}%
and then $S:=X_{n}\cap Y_{n}$ is a stable set. As a matter of fact, we
obtain in this way the worker-optimal stable set. Similarly, we may
construct the firm-optimal stable set by the same recurrence relations if we
start from $(X_{0},Y_{0}):=\varnothing \times X$. \medskip 

(b) If we start with $(X_{0},Y_{0}):=X\times \varnothing $, then we obtain $%
X_{1}=X\setminus R_{F}(\varnothing )=X$ and therefore 
\begin{equation}
X_{0}=X_{1},\quad Y_{1}=Y_{2},\quad X_{2}=X_{3},\quad Y_{3}=Y_{4},\ldots .
\label{62}
\end{equation}%
This implies that the above algorithm is equivalent to the more economical
Gale--Shapley algorithm. There we start with $X_{0}:=X$ and we compute
successively 
\begin{equation*}
Y_{1},X_{2},Y_{3},X_{4},\ldots 
\end{equation*}%
by using the recursive formulae 
\begin{equation*}
Y_{n+1}:=(X\setminus X_{n})\cup C_{W}(X_{n})\quad \text{and}\quad
X_{n+1}:=(X\setminus Y_{n})\cup C_{F}(Y_{n}).
\end{equation*}%
We stop when we obtain $X_{n-1}=X_{n+1}$ for the first time, and we set $%
S=C_{W}(X_{n-1})$. Indeed, the equalities \eqref{62} and $X_{n-1}=X_{n+1}$
imply that 
\begin{equation*}
X_{n-1}=X_{n}=X_{n+1}=X_{n+2}=\cdots \quad \text{and}\quad
Y_{n}=Y_{n+1}=Y_{n+2}=Y_{n+3}\cdots .
\end{equation*}%
Therefore $(X_{n+1},Y_{n+1})=(X_{n},Y_{n})$, and 
\begin{equation*}
S=X_{n}\cap Y_{n}=X_{n-1}\cap \left( (X\setminus X_{n-1})\cup
C_{W}(X_{n-1})\right) =C_{W}(X_{n-1}).
\end{equation*}%
In the last step we used that $C_{W}(X_{n-1})\subset X_{n-1}$ because $C_{W}$
is a choice map.

Analogously, we may construct the firm-optimal stable set by starting with $%
Y_{0}:=X$, computing successively $X_{1},Y_{2},X_{3},Y_{4},\ldots $ by the
same formulae as above, and setting $S=C_{F}(Y_{n-1})$ for the first $n$
such that $Y_{n-1}=Y_{n+1}$.
\end{remark}

\bigskip 

\begin{examples}
\label{e63}\mbox{}

(a) We cannot replace the revealed preference condition with the substitutes
condition in Theorem \ref{t37}. To show this consider the choice maps $%
C_{W}:=C_{1}$ and $C_{F}:=C_{2}$ of Example \ref{e61} (a) on the set $%
X=\{a,b\}$. Then $C_{W}$ is revealing and $C_{F}$ is persistent. However,
there is no stable set. Indeed, we have $C_{W}(S)=S=C_{F}(S)$ only if $%
S=\varnothing $ or $S=\{b\}$, so that only these two sets are individually
rational (see Remark \ref{r32}). However, $S=\varnothing $ is blocked by $%
\{b\}$ because 
\begin{equation*}
C_{W}(S\cup \{b\})=C_{F}(S\cup \{b\})=\{b\}\neq S,
\end{equation*}%
and $S=\{b\}$ is blocked by $\{a\}$ because 
\begin{equation*}
C_{W}(S\cup \{a\})=\{a\}\neq S\quad \text{and}\quad C_{F}(S\cup
\{a\})=\varnothing \neq S.
\end{equation*}%
Hence none of these sets is stable. \medskip 

(b) We cannot replace the revealed preference condition with the consistence
in Theorem \ref{t37} either. To show this consider the choice maps $%
C_{W}:=C_{1}$ and $C_{F}:=C_{3}$ of Example \ref{e61} (a) on the set $%
X=\{a,b\}$. Then $C_{W}$ is revealing and $C_{F}$ is consistent. However,
there is no stable matching. Indeed, we have $C_{W}(S)=S=C_{F}(S)$ only if $%
S=\varnothing $, so this is the only individually rational set. For $%
S=\varnothing $ the condition \eqref{33} is satisfied only if $%
S_{W}=\varnothing $, and then $S_{F}=X$ by \eqref{32}. However, then $%
C_{F}(S_{F})=X\neq S$, so that \eqref{34} fails.
\end{examples}

\bigskip 

Now we turn to the proofs of Theorems \ref{t310} and \ref{t313}. They are
independent of the preceding part of the present section.

\begin{proof}[Proof of Theorem \protect\ref{t310}]
The choice map $C$ remains the same if we change each $Y_n$ to $Y_n\cap Y$
in the construction. The choice map does not change either if we complete
the family $\{Y_n\}$ with $Y^{\prime }:=Y\setminus\cup Y_n$ corresponding to
the quota $q^{\prime }:=\card Y^{\prime }$. Without loss of generality we
assume henceforth that $\{Y_n\}$ is a \emph{partition} of $Y$, i.e., $Y$ is
the \emph{disjoint union} of the sets $Y_n$.

Let $A,B\subset X$ be two sets satisfying $C(A)\subset B$; we have to show
that if $y_k\in A\cap C(B)$ for some $k$, then $y_k\in C(A)$.

First we establish by induction on $j$ the following inequalities: 
\begin{equation}  \label{63}
\card (C_j(A)\cap Y_n)\le \card (C_j(B)\cap Y_n)\text{ for all }n,\quad
j=0,\ldots, k-1.
\end{equation}

For $j=0$ our claim reduces to the trivial equality $0=0$. Assuming that the
inequalities hold until some $j<k-1$, consider the (unique) index $m$ for
which $y_{j+1}\in Y_m$. For each $n\ne m$ we have 
\begin{equation*}
C_j(A)\cap Y_n=C_{j+1}(A)\cap Y_n\text{ and }C_j(B)\cap Y_n=C_{j+1}(B)\cap
Y_n
\end{equation*}
and therefore 
\begin{equation*}
\card (C_{j+1}(A)\cap Y_n)\le \card (C_{j+1}(B)\cap Y_n)
\end{equation*}
by our induction hypothesis. For $n=m$ the only critical case is when 
\begin{equation*}
y_{j+1}\in C_{j+1}(A)\setminus C_{j+1}(B).
\end{equation*}
Since $y_{j+1}\in C(A)$ implies $y_{j+1}\in B$ and since 
\begin{equation*}
\card C_j(B)\le \card C_{k-1}(B)\le q-1
\end{equation*}
because $y_k\in C(B)$ and therefore 
\begin{equation*}
\card C_{k-1}(B)=\card C_k(B)-1\le q-1,
\end{equation*}
by the construction this can only happen if 
\begin{equation*}
\card (C_j(A)\cap Y_m)\le q_m-1\text{ and }\card (C_j(B)\cap Y_m)=q_m.
\end{equation*}
But then we have 
\begin{align*}
\card (C_{j+1}(A)\cap Y_m) &=\card (C_j(A)\cap Y_m)+1 \\
&\le q_m \\
&=\card (C_j(B)\cap Y_m) \\
&=\card (C_{j+1}(B)\cap Y_m)
\end{align*}
as required.

Since $y_k\in A\cap C(B)$, we have $y_k\in A$. Furthermore, since $%
C(A)\subset Y$ and the sets $Y_n$ form a partition of $Y$, it follows from %
\eqref{63} that 
\begin{align*}
\card C_{k-1}(A) &=\cup_{n}\card (C_{k-1}(A)\cap Y_n) \\
&\le \cup_{n}\card (C_{k-1}(B)\cap Y_n) \\
&= \card C_{k-1}(B) \\
&=\card C_k(B)-1 \\
&\le q-1
\end{align*}
because $C_k(B)\setminus C_{k-1}(B)=\{y_k\}$.

Furthermore, in case $y_k\in Y_n$ we have 
\begin{equation*}
(C_k(B)\cap Y_n)\setminus (C_{k-1}(B)\cap Y_n)=\{y_k\}
\end{equation*}
and therefore 
\begin{align*}
\card\left( C_{k-1}(A)\cap Y_n\right) &\le\card \left(C_{k-1}(B)\cap
Y_n\right) \\
&=\card \left(C_k(B)\cap Y_n\right)-1 \\
&\le q_n-1.
\end{align*}

Summarizing, the conditions \eqref{313}--\eqref{315} are satisfied and we
conclude that $y_{k}\in C(A)$ by construction. This completes the proof.
\end{proof}

\begin{remark}
\label{r64}\mbox{}

(a) The choice map constructed in Theorem \ref{t310} is consistent even if
the sets $Y_n\cap Y$ are not disjoint. Indeed, if $C(A)\subset B\subset A$,
then comparing the construction of 
\begin{equation*}
C_0(A)\subset C_1(A)\subset\cdots\text{ and }C_0(B)\subset
C_1(B)\subset\cdots,
\end{equation*}
we see that $C_k(A)=C_k(B)$ for every $k$ and therefore $C(A)=C(B)$. The
equality $C_k(A)=C_k(B)$ is obvious for $k=0$ because both sides are equal
to zero. If it is true for some $k-1\ge 0$, then we have $y_k\in C_k(A)$ if
and only if 
\begin{align*}
&y_k\in A, \\
&\card\left( C_{k-1}(A)\cup\{y_k\}\right) \le q \\
&\card\left( (C_{k-1}(A)\cup\{y_k\})\cap Y_n\right) \le q_n\text{ for all }n,
\end{align*}
and $y_k\in C_k(B)$ if and only if 
\begin{align*}
&y_k\in B, \\
&\card\left( C_{k-1}(B)\cup\{y_k\}\right) \le q \\
&\card\left( (C_{k-1}(B)\cup\{y_k\})\cap Y_n\right) \le q_n\text{ for all }n.
\end{align*}
Since $C_{k-1}(A)=C_{k-1}(B)$ by the induction hypothesis, the equality $%
C_k(A)=C_k(B)$ will follow if we show that $y_k\in A\Longleftrightarrow
y_k\in B$ if the last two conditions are satisfied. Since $C(A)\subset
B\subset A$ and since $C(B)\subset B$ ($C$ is a choice map), we have 
\begin{align*}
\intertext{and}
&y_k\in B\Longrightarrow y_k\in C(B)\Longrightarrow y_k\in B\Longrightarrow
y_k\in A.
\end{align*}

\medskip

(b) The range of an idempotent choice map coincides with the set of its
fixed points: 
\begin{equation*}
\{C(A)\ :\ A\subset X\}=\{A\subset X\ :\ C(A)=A\}.
\end{equation*}
\end{remark}

\begin{example}
\label{e65} In Example \ref{e63} (b) the consistent choice map $C_F$ cannot
be obtained by the construction of Theorem \ref{t310} without the
disjointness condition (see Remark \ref{r64} (a)). A stronger counterexample
is the following. We consider a three-point set $X=\{a,b,c\}$ and the
following two choice maps:

\begin{tabular}{c||c|c|c|c|c|c|c|c}
$A$ & $\varnothing$ & $\{a\}$ & $\{b\}$ & $\{c\}$ & $\{a,b\}$ & $\{a,c\}$ & $%
\{b,c\}$ & $\{a,b,c\}$ \\ \hline\hline
$C_W(A)$ & $\varnothing$ & $\{a\}$ & $\{b\}$ & $\{c\}$ & $\{a,b\}$ & $\{c\}$
& $\{b,c\}$ & $\{b,c\}$ \\ \hline
$C_F(A)$ & $\varnothing$ & $\{a\}$ & $\{b\}$ & $\{c\}$ & $\{a\}$ & $\{a,c\}$
& $\{b\}$ & $\{a,c\}$%
\end{tabular}

Both choice maps are defined by the construction of Theorem \ref{t310}. For $%
C_W$ we take $Y=X$ with the preference order $c\succ b\succ a$ and quota $q=2
$, and we set $Y_1=\{a,c\}$ with the quota $q_1=1$. This is a revealing
choice map. The choice map $C_F$ is the one given in Example \ref{e312}
above: a consistent but not revealing choice map because the disjointness
condition is not satisfied.

In order to find a stable set $S$ we have to cover $X=\{a,b,c\}$ by two sets 
$S_W$ and $S_F$ satisfying $C_W(S_W)=S=C_W(S)$ and $C_F(S_F)=S=C_F(S)$. The
equalities $C_W(S)=S=C_F(S)$ are satisfied if and only if $S$ has at most
one element, so that there are four candidates for the stable set $S$. We
can see easily from the table that in order to have $C_W(S_W)=S=C_W(S)$,

\begin{itemize}
\item in case $S=\varnothing$ we must have $S_W=S_F=\varnothing$;

\item in case $S=\{a\}$ we must have $S_W=\{a\}$ and $S_F\subset\{a,b\}$;

\item in case $S=\{b\}$ we must have $S_W=\{b\}$ and $S_F\subset\{b,c\}$;

\item in case $S=\{c\}$ we must have $S_W\subset\{a,c\}$ and $S_F=\{c\}$.
\end{itemize}

Since $S_W\cup S_F\ne X$ in all these cases, we conclude that there is no
stable set.
\end{example}

\begin{proof}[Proof of Theorem \protect\ref{t313}]
If $C(A)\subset B$, then setting $A_i:=A\cap X_i$ and $B_i:=B\cap X_i$ we
have 
\begin{align*}
A\cap C(B)\subset C(A) &\Longleftrightarrow \left( A\cap C(B)\right) \cap
X_i\subset C(A)\cap X_i\text{ for all }i \\
&\Longleftrightarrow A_i\cap C_i(B_i)\subset C_i(A_i)\text{ for all }i. %
\qedhere
\end{align*}
\end{proof}

\end{document}